\documentclass[preprint,11pt]{elsarticle}

\usepackage{lmodern}
\usepackage{amsmath,amssymb,amsthm}
\usepackage[margin=1in]{geometry}
\usepackage{microtype}
\usepackage{hyperref}

\numberwithin{equation}{section}
\newtheorem{theorem}{Theorem}[section]
\newtheorem{lemma}[theorem]{Lemma}
\newtheorem{proposition}[theorem]{Proposition}
\newtheorem{corollary}[theorem]{Corollary}
\theoremstyle{definition}

\newtheorem{example}[theorem]{Example}
\newtheorem{remark}[theorem]{Remark}

\newcommand{\R}{\mathbb{R}}
\newcommand{\C}{\mathbb{C}}

\newcommand{\RZ}{\mathcal{RZ}}
\newcommand{\UC}{\mathcal{UC}}
\DeclareMathOperator{\Arg}{Arg}

\begin{document}
	
	\begin{frontmatter}
		
		\title{Interlacing on the Unit Circle via Coefficientwise Reciprocals}
		
		\author[a]{Jianxi Mao}
		\ead{maojx@dlut.edu.cn}
		
		\author[a]{Lijie Wang}
		\ead{lijiewang26@hotmail.com}
		
		\author[b]{Sainan Zheng\corref{cor1}}
		\ead{zhengsainan@dufe.edu.cn}
		\cortext[cor1]{Corresponding author}
		
		\address[a]{School of Mathematical Sciences, Dalian University of Technology,
			Dalian 116024, P. R. China}
		
		\address[b]{School of Data Science and Artificial Intelligence,
			Dongbei University of Finance and Economics,
			Dalian 116025, P. R. China}\begin{abstract}
			Let
			$
			f(z)=\sum_{k=0}^{n}a_kz^k
			$
			be a polynomial with positive coefficients, and define its coefficientwise
			reciprocal by
			$
			f^{\#}(z)=\sum_{k=0}^{n}\frac{z^k}{a_k}.
			$
			It is known that if $f$ is palindromic and has only negative real zeros, then all zeros of
			$f^{\#}$ lie on the unit circle.
			We prove that if $p$ and $q$ are palindromic polynomials of degrees $n$ and $n+1$, respectively, with only negative real zeros, then
			their coefficientwise reciprocals have only simple zeros,
			and $p^{\#}$ strictly interlaces $q^{\#}$ on the unit circle.
			Our proof is based on finite Blaschke
			products and the comparison of their boundary phases. As an immediate
			consequence, we obtain strict interlacing for the reciprocal binomial,
			reciprocal Eulerian, and reciprocal Narayana polynomials.
			By combining a sign change in the $\gamma$-coefficients
			with coefficientwise reciprocation, we further construct
			strictly interlacing families from Rogers--Szeg\H{o},
			Poupard, and Kreweras-related polynomials.
		\end{abstract}

		
	\end{frontmatter}

	\section{Introduction}

	Determining whether all zeros of a polynomial lie in
	a prescribed region of the complex plane or on its
	boundary is a classical problem in the theory of
	polynomials~\cite{Marden1966,RahmanSchmeisser2002}.
	Such questions also motivate the study of transformations that preserve
	specified zero locations~\cite{BorceaBranden2009Circular}.

	Besides the location of the zeros of an individual polynomial,
	another important question concerns the interlacing of zeros.
	A classical example follows from Rolle's theorem:
	if $f$ has only real zeros, then the zeros of $f'$ interlace those of $f$.
	The interlacing property is also a fundamental tool for
	studying polynomial sequences with only real zeros~\cite{Brenti1989,LiuWang2007,Obreschkoff1963}.
	
	Interlacing problems have also been studied for polynomials whose zeros lie
	on the unit circle.
	For example, when the parameter lies in $[0,1]$, 
	the Rogers--Szeg\H{o} polynomials form a classical
	palindromic family with zeros on the unit circle through their connection with the continuous $q$-Hermite polynomials; see \cite{Ismail2009}.
	Lakatos and Losonczi \cite{LakatosLosonczi2007} obtained circular
	interlacing results for reciprocal polynomials, while Simon
	\cite{Simon2007} studied interlacing properties of paraorthogonal
	polynomials on the unit circle. In this paper, we study a different source
	of interlacing on the unit circle, arising from coefficientwise
	reciprocation.

	We say a polynomial
	$
	f(z)=a_0+a_1z+\cdots+ a_nz^n
	$
	is \emph{palindromic} if
	$
	a_k=a_{n-k}
	$
	for all $k$.
	Write $f\in\RZ$ if all zeros of $f$ are real, and write $P\in\UC$ if all
	zeros of $P$ lie on the unit circle
	$
	\{z\in\C:|z|=1\}.
	$
	In particular, if $P\in\UC$ has positive coefficients, then $P$ is palindromic. We define
	\[
	\mathcal P_n=
	\left\{
	f(z)=\sum_{k=0}^{n}a_kz^k
	\;\middle|\; f\textrm{ is palindromic and }
	a_k>0 \textrm{ for } 0\le k\le n
	\right\}.
	\]
	Clearly, if $f\in\mathcal P_n\cap\RZ$, then all zeros of $f$ are negative.
	A classical example is the binomial polynomial $(1+z)^n$.

	Suppose that $P,Q\in\UC$ and  $\deg Q=\deg P+1=n+1$. Let
	\[
	e^{i\theta_1},\ldots,e^{i\theta_n}
	\quad\text{and}\quad
	e^{i\tau_1},\ldots,e^{i\tau_{n+1}}
	\]
	be the zeros of $P$ and $Q$, respectively, where the arguments are ordered increasingly
	in $[0,2\pi)$.
	We say that $P$
	\emph{interlaces} $Q$ if
	\begin{equation}\label{eq:interlace}
		0\le\tau_1\le\theta_1\le\tau_2\le\cdots\le\tau_n\le\theta_n
		\le\tau_{n+1}<2\pi.
	\end{equation}
	We say that $P$ \emph{strictly interlaces} $Q$ if all inequalities in
	\eqref{eq:interlace} are strict.

	Let
	$
	f\in \mathcal P_n.
	$
	Define the \emph{coefficientwise reciprocal polynomial} of $f$ by
	\[
	f^{\#}(z):=\sum_{k=0}^{n}\frac{z^k}{a_k}.
	\]
	It is known that if $f\in\mathcal P_n\cap\RZ$, then
	$f^{\#}\in\UC$; see \cite{Chen1995} and \cite[Theorem~3.7]{ChenMaoWangXiao2026}.
	In this paper, we investigate whether coefficientwise reciprocation
	also produces interlacing on the unit circle for independently
	chosen polynomials of consecutive degrees.
	Our main result is the following.
	
	\begin{theorem}\label{thm:main}
		Let
		\[
		p(z)=\sum_{k=0}^{n}p_kz^k\in\mathcal P_n \cap \RZ,
		\qquad
		q(z)=\sum_{k=0}^{n+1}q_kz^k\in\mathcal P_{n+1} \cap \RZ,
		\]
		where $n\ge1$. Then
		\begin{enumerate}[(i)]
			\item $p^{\#},q^{\#}\in\UC$,  and all zeros of $p^{\#}$ and $q^{\#}$ are simple;
			\item
			$p^{\#}$ strictly interlaces $q^{\#}$.
		\end{enumerate}
	\end{theorem}

	The paper is organized as follows.
	In Section~\ref{Basic properties}, we collect the properties of
	coefficientwise reciprocal polynomials, finite Blaschke products, and
	half-polynomials needed in the proof.
	Section~\ref{proof of thm} establishes the required coefficient and
	sector estimates and proves Theorem~\ref{thm:main}.
	Section~\ref{sec:applications} presents applications to classical
	families of polynomials with only real zeros and to families obtained through the
	$\gamma$-sign transform.
	Finally, Section~\ref{Remark} discusses the
	Rodr\'iguez--Villegas transform and the resulting interlacing properties
	on the line $\operatorname{Re}z=-1/2$.

	\section{Basic properties}\label{Basic properties}

	We first present the basic properties needed in the sequel.
	When proving Theorem~\ref{thm:main}, we may assume that the
	polynomials are monic, i.e., the leading coefficient is 1.
	
	Let $a_0,a_1,\ldots,a_n$ be a sequence of positive real numbers.
	The sequence is said to be \emph{unimodal} if there exists an index
	$0\le m\le n$ such that
	$a_0\le\cdots\le a_{m-1}\le a_m\ge a_{m+1}\ge\cdots\ge a_n$.
	The sequence is said to be \emph{log-concave} if
	$a_k^2\ge a_{k-1} a_{k+1}$ for $k=1,\ldots,n-1$.
	Clearly, log-concavity implies unimodality.
	
	The Newton inequalities state that
	if
	$
	f(z)=\sum_{k=0}^n a_kz^k\in\RZ,
	$
	then
	\begin{equation}\label{Newton}
		a_k^2
		\ge
		a_{k-1}a_{k+1}
		\left(1+\frac{1}{k}\right)
		\left(1+\frac{1}{n-k}\right)
	\end{equation}
	for $1\le k\le n-1$.
	In particular, if $f\in\mathcal P_n\cap\RZ$, then
	$(a_k)_{k=0}^n$ is strictly log-concave. Since the coefficient
	sequence is symmetric, it increases strictly up to the middle:
	\begin{equation}\label{eq:coefincrease}
		0<a_0<a_1<\cdots<a_{\lfloor n/2\rfloor}.
	\end{equation}
	
	\subsection{Finite Blaschke products}
	We first recall the notion of a continuous argument.
	Let $S\subseteq\mathbb C$ and let
	\[
	\varphi:S\longrightarrow\mathbb C\setminus\{0\}
	\]
	be continuous. A continuous function
	$\alpha:S\longrightarrow\mathbb R$ is called a
	\emph{continuous argument} of $\varphi$ if
	\[
	\varphi(z)=|\varphi(z)|e^{i\alpha(z)}
	\qquad\text{for all }z\in S.
	\]
	We use $\Arg z$ to denote the principal argument of $z$.
	
	Finite Blaschke products  provide a natural analytic tool for studying
	zeros on the unit circle~\cite{GarciaMashreghiRoss2017,MandelRobins2019,Simanek2020}.
	A finite \emph{Blaschke product} of degree $n$ is a rational function
	of the form
	\[
	B(z)=\prod_{k=1}^{n}\frac{z-w_k}{1-\overline{w_k}z},
	\qquad |w_k|<1\quad(1\le k\le n).
	\]
	The zeros are counted with multiplicity; in particular, $w_k=0$ is
	allowed.
	
	We shall use the following standard properties of finite Blaschke
	products.
	
	\begin{lemma}[\cite{GarciaMashreghiRoss2017}]\label{lem:blaschke}
		Let $B$ be a finite Blaschke product of degree $n$. Then
		\[
		|B(e^{it})|=1,
		\qquad t\in\R.
		\]
		If
		$
		B(e^{it})=e^{i\alpha(t)}
		$
		for a continuous argument $\alpha$, then $\alpha$ is strictly increasing
		and
		\[
		\alpha(2\pi)-\alpha(0)=2\pi n.
		\]
	\end{lemma}

	\subsection{Half-polynomial decomposition}
	We now present a half-polynomial decomposition of a palindromic polynomial.
	Let
	\[
	f(z)=\sum_{k=0}^{n}a_kz^k
	\]
	be palindromic. If $n=2m$, then
	\[
	\begin{aligned}
		f(z)
		={}&\left(a_0+\cdots+a_{m-1}z^{m-1}+\frac{a_m}{2}z^m\right)
		+\left(\frac{a_m}{2}z^m+a_{m-1}z^{m+1}+\cdots+a_0z^{2m}\right).
	\end{aligned}
	\]
	If $n=2m+1$, then
	\[
	\begin{aligned}
		f(z)
		={}&\left(a_0+\cdots+a_{m-1}z^{m-1}+a_mz^m\right)
		+\left(a_mz^{m+1}+a_{m-1}z^{m+2}+\cdots+a_0z^{2m+1}\right).
	\end{aligned}
	\]
	Define the \emph{half-polynomial} of $f(z)$ by
	\begin{equation}\label{eq:half}
		H[f](z)=
		\begin{cases}
			\displaystyle \frac{a_m}{2}+a_{m-1}z+\cdots+a_0z^m,
			&n=2m,\\[6pt]
			\displaystyle a_m+a_{m-1}z+\cdots+a_0z^m,
			&n=2m+1.
		\end{cases}
	\end{equation}
	For a real polynomial $g$ of degree $n$, define its
	\emph{reversed polynomial} by
	\begin{equation}\label{eq:star}
		g^*(z):=z^ng\left(\frac1z\right).
	\end{equation}
	Thus $g^*=g$ if and only if $g$ is palindromic.
	With this notation, the half-polynomial decomposition takes the form
	\begin{equation}\label{eq:half-decomp}
		f(z)=
		\begin{cases}
			\bigl(H[f]\bigr)^*(z)+z^mH[f](z),&n=2m,\\[4pt]
			\bigl(H[f]\bigr)^*(z)+z^{m+1}H[f](z),&n=2m+1.
		\end{cases}
	\end{equation}
	
	We first recall the Enestr\"om--Kakeya theorem, which will be applied to the half-polynomial of \(f^\#\).
	A polynomial is called \emph{Schur stable} if all its zeros lie in $\{z\in\C:|z|<1\}$.
	\begin{lemma}[\cite{GardnerGovil2014}]\label{lem:enestrom}
		Let
		$
		g(z)=\sum_{k=0}^{n}c_kz^k
		$
		be a polynomial with real coefficients. If
		\[
		0<c_0<c_1<\cdots<c_n,
		\]
		then $g$ is Schur stable.
	\end{lemma}
	We apply this result to the half-polynomial of $f^{\#}$.
	\begin{lemma}\label{lem:half-stable}
		Let $f\in\mathcal P_n\cap\RZ$. Then $H[f^{\#}]$ is Schur stable.
		Moreover, all zeros of $\bigl(H[f^{\#}]\bigr)^*$ lie in
		$\{z:|z|>1\}$.
	\end{lemma}
	
	\begin{proof}
		Let
		$
		f(z)=\sum_{k=0}^{n}a_kz^k.
		$
		By \eqref{eq:coefincrease},
		\[
		0<a_0<a_1<\cdots<a_{\lfloor n/2\rfloor}.
		\]
		Write
		\[
		f^{\#}(z)=\sum_{k=0}^{n}b_kz^k,
		\qquad b_k=\frac1{a_k}.
		\]
		Then
		\[
		b_0>b_1>\cdots>b_{\lfloor n/2\rfloor}>0.
		\]
		Hence, by \eqref{eq:half}, the coefficients of $H[f^{\#}]$ in ascending powers of $z$ are strictly increasing.
		It follows from Lemma~\ref{lem:enestrom} that $H[f^{\#}]$ is
		Schur stable.
		
		If $w\ne 0$ is a zero of $H[f^{\#}]$, then $1/w$ is a zero of
		$\bigl(H[f^{\#}]\bigr)^*$ by~\eqref{eq:star}. Since $|w|<1$, all zeros of
		$\bigl(H[f^{\#}]\bigr)^*$ lie in $\{z:|z|>1\}$.
	\end{proof}
	
	

	\section{Proof of Theorem \ref{thm:main} }\label{proof of thm}
	In this section, we will prove Theorem~\ref{thm:main}.
	The case $n=1$ of Theorem~\ref{thm:main} is immediate.
	Hence, in what follows, we assume that $n\ge2$.
	We first establish the estimates needed for the proof.

	\subsection{Coefficient estimates}
	\begin{lemma}\label{lem:ratio}
		Suppose
		$
		f(z)=\sum_{k=0}^{n}a_kz^k\in\mathcal P_n\cap\RZ
		$
		with $a_0=1.$
		Let $m=\lfloor n/2\rfloor$ and
		\[
		\bigl(H[f^{\#}]\bigr)^*(z)=\sum_{k=0}^{m}c_kz^k.
		\]
		Then
		$
		c_0=1,$ and $c_1\le1/n.
		$
		For $2\le k\le m$, we have
		\[
		\frac{c_{k-1}}{c_k}\ge\frac{k}{k-1}.
		\]
		In particular, if $m\ge2$, then
		\[
		c_1\ge2c_2\ge\cdots\ge mc_m>0.
		\]
	\end{lemma}
	
	\begin{proof}
		By \eqref{eq:star} and \eqref{eq:half},
		\[
		\bigl(H[f^{\#}]\bigr)^*(z)=
		\begin{cases}
			\displaystyle
			\frac1{a_0}+\frac1{a_1}z+\cdots+\frac1{a_{m-1}}z^{m-1}
			+\frac1{2a_m}z^m,
			&n=2m,\\[8pt]
			\displaystyle
			\frac1{a_0}+\frac1{a_1}z+\cdots+\frac1{a_{m-1}}z^{m-1}
			+\frac1{a_m}z^m,
			&n=2m+1.
		\end{cases}
		\]
		Since $a_0=1$, we have $c_0=1$.
		
		Since $f\in\mathcal P_n\cap\RZ$, all zeros of $f$ are negative. Hence
		\[
		f(z)=\prod_{j=1}^{n}(z+r_j),\qquad r_j>0.
		\]
		Since $a_0=1$, we have
		$
		\prod_{j=1}^{n}r_j=1.
		$
		Therefore, by the arithmetic--geometric mean inequality,
		\[
		a_1=\sum_{j=1}^{n}\frac1{r_j}
		\ge n\left(\prod_{j=1}^{n}\frac1{r_j}\right)^{1/n}=n.
		\]
		Hence $c_1\le 1/n$. Indeed, $c_1=1/a_1$ for $n\ge3$, while for
		$n=2$ we have
		\[
		c_1=\frac{1}{2a_1}\le\frac14<\frac12.
		\]
		
		By Newton's inequalities, the normalized coefficients
		\[
		\frac{a_k}{\binom{n}{k}},\qquad 0\le k\le n,
		\]
		form a log-concave sequence.
		Since $f$ is palindromic, this sequence is symmetric and hence nondecreasing up to the middle. Therefore,
		\[
		\frac{a_k}{a_{k-1}}
		\ge
		\frac{\binom{n}{k}}{\binom{n}{k-1}}
		=
		\frac{n-k+1}{k},
		\qquad
		1\le k\le m.
		\]
		If $2\le k\le m$ and $k$ is not the middle index when $n$ is even, then
		\[
		\frac{c_{k-1}}{c_k}
		=\frac{a_k}{a_{k-1}}
		\ge\frac{n-k+1}{k}
		\ge\frac{k}{k-1}.
		\]
		If $n=2m$ and $k=m$, then
		\[
		\frac{c_{m-1}}{c_m}
		=\frac{2a_m}{a_{m-1}}
		\ge\frac{2(m+1)}m
		\ge\frac{m}{m-1}.
		\]
		Thus
		\[
		\frac{c_{k-1}}{c_k}\ge\frac{k}{k-1},
		\qquad 2\le k\le m,
		\]
		which is equivalent to
		\[
		c_1\ge2c_2\ge\cdots\ge mc_m>0.
		\]
		This completes the proof.
	\end{proof}

	\subsection{A sector bound}
	Define
	\begin{equation}\label{eq:Bf}
		\mathcal B_f(z):=
		z^{\lceil n/2\rceil}
		\frac{H[f^{\#}](z)}{\bigl(H[f^{\#}]\bigr)^*(z)},
		\qquad f\in\mathcal P_n\cap\RZ.
	\end{equation}
	By Lemma~\ref{lem:half-stable}, $\mathcal B_f$ is a finite Blaschke
	product of degree $n$. Moreover, \eqref{eq:half-decomp} gives
	\begin{equation}\label{eq:root-level}
		f^{\#}(z)=
		\bigl(H[f^{\#}]\bigr)^*(z)
		\bigl(1+\mathcal B_f(z)\bigr).
	\end{equation}
	Since $\bigl(H[f^{\#}]\bigr)^*(z)$ has no zeros on the unit circle by Lemma~\ref{lem:half-stable},
	for $0\le t<2\pi$, $f^{\#}(e^{it})=0$ if and only if
	$\mathcal B_f(e^{it})=-1.$
	We next use Lemma~\ref{lem:ratio} to obtain the following sector bound.
	\begin{lemma}\label{lem:sector}
		Let $f\in\mathcal P_n\cap\RZ$ be monic. Then
		\begin{equation}\label{eq:sector}
			0<\Arg\bigl((H[f^{\#}])^*(e^{it})\bigr)<\frac t2,
			\qquad 0<t<\pi.
		\end{equation}
	\end{lemma}

	\begin{proof}
		Write
		\[
		\bigl(H[f^{\#}]\bigr)^*(z)=\sum_{k=0}^{m}c_kz^k,
		\qquad m=\left\lfloor\frac n2\right\rfloor.
		\]
		By Lemma~\ref{lem:ratio},
		\begin{equation}\label{eq:ck-chain}
			\frac1n\ge c_1\ge2c_2\ge\cdots\ge mc_m>0.
		\end{equation}
		Hence
		\[
		\sum_{k=1}^{m}c_k\le mc_1\le\frac mn\le\frac12.
		\]
		Since $c_0=1$,
		\[
		\bigl(H[f^{\#}]\bigr)^*(e^{it})
		=1+\sum_{k=1}^{m}c_k(\cos (kt)+i\sin (kt)).
		\]
		Therefore,
		$
		\operatorname{Re}\bigl((H[f^{\#}])^*(e^{it})\bigr)
		\ge1-\sum_{k=1}^{m}c_k\ge\frac12.
		$
		
		For the imaginary part, the classical Fej\'er inequality (see~\cite[Eq.~(1.1)]{BY01}) gives
		\[
		S_j(t):=\sum_{k=1}^{j}\frac{\sin (kt)}{k}>0,
		\qquad j\ge1,\quad 0<t<\pi.
		\]
		Put $d_k=kc_k$ for $1\le k\le m$ and $d_{m+1}=0$.
		By \eqref{eq:ck-chain},
		$d_k-d_{k+1}\ge0$ and $d_m-d_{m+1}>0$.
		Abel summation gives
		\[
		\begin{aligned}
			\operatorname{Im}\bigl((H[f^{\#}])^*(e^{it})\bigr)
			&=\sum_{k=1}^{m}c_k\sin (kt)\\
			&=\sum_{j=1}^{m}(d_j-d_{j+1})S_j(t)>0,
			\qquad 0<t<\pi.
		\end{aligned}
		\]
		Set $R(t) = |(H[f^{\#}])^*(e^{it})|$ and $\vartheta(t) = \Arg ((H[f^{\#}])^*(e^{it}))$. Since both the real and imaginary parts of $(H[f^{\#}])^*(e^{it})$ are positive, we have
		$R(t) > 0 $ and $ 0 < \vartheta(t) < \pi/2$.
		
		It remains to prove the upper bound in \eqref{eq:sector}.
		By \eqref{eq:ck-chain}, $kc_k\le c_1\le1/n$, and hence
		\[
		\sum_{k=1}^{m}(2k-1)c_k
		=2\sum_{k=1}^{m}kc_k-\sum_{k=1}^{m}c_k
		\le\frac{2m}{n}-\sum_{k=1}^{m}c_k<1.
		\]
		Using $|\sin(\ell x)|\le\ell\sin x$ for integers $\ell\ge1$ and
		$0<x<\pi$, we obtain
		\[
		\begin{aligned}
			R(t)\sin\left(\frac t2-\vartheta(t)\right)
			&=\sin\frac t2-
			\sum_{k=1}^{m}c_k\sin\frac{(2k-1)t}{2}\\
			&\ge
			\left(1-\sum_{k=1}^{m}(2k-1)c_k\right)
			\sin\frac t2>0.
		\end{aligned}
		\]
		Since
		\[
		-\frac\pi2<\frac t2-\vartheta(t)<\frac\pi2,
		\]
		we have $t/2-\vartheta(t)>0$. Hence
		\[
		0<\vartheta(t)<\frac t2,
		\]
		which proves \eqref{eq:sector}.
	\end{proof}

	\subsection{Phase separation}
	
	We now use the sector bound in Lemma~\ref{lem:sector} to compare the finite
	Blaschke products associated with polynomials of consecutive degrees.
	The following result shows that these products take distinct
	values at every point of the unit circle except $1$.

	\begin{proposition}\label{prop:noncoincidence}
		Let $p\in\mathcal P_n\cap\RZ$ and
		$q\in\mathcal P_{n+1}\cap\RZ$. Then
		\[
		\mathcal B_p(e^{it})\ne\mathcal B_q(e^{it}),
		\qquad 0<t<2\pi.
		\]
	\end{proposition}

	\begin{proof}
		We first consider the case
		$
		\deg p=2m,\deg q=2m+1.
		$
		Then
		\begin{align*}
			\begin{cases}
				\bigl(H[p^{\#}]\bigl)^*(e^{it})
				=\displaystyle\frac{1}{p_0}
				+\displaystyle\frac{1}{p_1}e^{it}
				+\cdots
				+\displaystyle\frac{1}{p_{m-1}}e^{i(m-1)t}
				+\displaystyle\frac{1}{2p_m}e^{imt},
				\\[6pt]
				\bigl(H[q^{\#}]\bigl)^*(e^{it})
				=\displaystyle\frac{1}{q_0}
				+\displaystyle\frac{1}{q_1}e^{it}
				+\cdots
				+\displaystyle\frac{1}{q_{m-1}}e^{i(m-1)t}
				+\displaystyle\frac{1}{q_m}e^{imt}.
			\end{cases}
		\end{align*}
		
		Write
		\[
		G_p(z):=\bigl(H[p^{\#}]\bigr)^*(z),
		\qquad
		G_q(z):=\bigl(H[q^{\#}]\bigr)^*(z).
		\]
		By \eqref{eq:star}, since both $H[p^{\#}]$ and $H[q^{\#}]$ have degree $m$,
		\[
		G_p(z)=z^mH[p^{\#}](1/z),
		\qquad
		G_q(z)=z^mH[q^{\#}](1/z).
		\]
		Hence, for $z=e^{it}$,
		\[
		G_p(e^{it})
		=e^{imt}H[p^{\#}](e^{-it}),
		\qquad
		G_q(e^{it})
		=e^{imt}H[q^{\#}](e^{-it}).
		\]
		Since $H[p^{\#}]$ and $H[q^{\#}]$ have real coefficients,
		\[
		H[p^{\#}](e^{-it})
		=\overline{H[p^{\#}](e^{it})},
		\qquad
		H[q^{\#}](e^{-it})
		=\overline{H[q^{\#}](e^{it})}.
		\]
		Therefore,
		\[
		H[p^{\#}](e^{it})
		=e^{imt}\overline{G_p(e^{it})},
		\qquad
		H[q^{\#}](e^{it})
		=e^{imt}\overline{G_q(e^{it})}.
		\]
		Hence, by \eqref{eq:Bf},
		\begin{equation}\label{eq:B-boundary}
			\mathcal B_p(e^{it})
			=e^{2mit}\frac{\overline{G_p(e^{it})}}{G_p(e^{it})},
			\qquad
			\mathcal B_q(e^{it})
			=e^{(2m+1)it}\frac{\overline{G_q(e^{it})}}{G_q(e^{it})}.
		\end{equation}

		For $0<t<\pi$, set
		\[
		\alpha(t)=\Arg G_p(e^{it}),
		\qquad
		\beta(t)=\Arg G_q(e^{it}).
		\]
		By Lemma~\ref{lem:sector},
		\[
		0<\alpha(t)<\frac t2,
		\qquad
		0<\beta(t)<\frac t2.
		\]
		Therefore,
		\[
		\frac{\mathcal B_q(e^{it})}{\mathcal B_p(e^{it})}
		=\exp\bigl(i(t+2\alpha(t)-2\beta(t))\bigr).
		\]
		The above inequalities imply
		\[
		0<t+2\alpha(t)-2\beta(t)<2t<2\pi.
		\]
		It follows that
		\[
		\mathcal B_p(e^{it})\ne\mathcal B_q(e^{it}),
		\qquad 0<t<\pi.
		\]
		
		Since $\mathcal B_p$ and $\mathcal B_q$ have real coefficients, the same
		conclusion holds for $\pi<t<2\pi$ by conjugation.
		The case
		$
		\deg p=2m+1, \deg q=2m+2
		$
		follows by the same argument.
		
		It remains to consider $t=\pi$. In this case,
		\[
		\mathcal B_p(-1)=(-1)^n,
		\qquad
		\mathcal B_q(-1)=(-1)^{n+1}.
		\]
		Thus
		\[
		\mathcal B_p(-1)\ne\mathcal B_q(-1),
		\]
		which completes the proof.
	\end{proof}

	\begin{proof}[Proof of Theorem~\ref{thm:main}]
		We first show that all zeros of \(f^\#\) are simple.
		Recall that by~\eqref{eq:Bf},
		$$	\mathcal B_f(z):=
		z^{\lceil n/2\rceil}
		\frac{H[f^{\#}](z)}{\bigl(H[f^{\#}]\bigr)^*(z)},
		\qquad f\in\mathcal P_n\cap\RZ.
		$$
		By~\eqref{eq:Bf} and~\eqref{eq:root-level}, $\mathcal B_f$ is a finite Blaschke product of
		degree $n$,
		and
		\[
		f^\#(z)=\bigl(H[f^{\#}]\bigr)^*(z)\bigl(1+\mathcal B_f(z)\bigr).
		\]
		The equation $\mathcal B_f(z)=-1$ has exactly $n$ distinct
		solutions on the unit circle by Lemma~\ref{lem:blaschke}. Each is a zero of $f^\#$.
		Since $\deg f^\#=n$, these are all its zeros, and they
		are simple.

		It remains to prove that $p^{\#}$ strictly interlaces  $q^{\#}$.
		By \eqref{eq:root-level},
		$p^{\#}(e^{it})=0$ if and only if
		$\mathcal B_p(e^{it})=-1,
		$
		and $q^{\#}(e^{it})=0$ if and only if
		$\mathcal B_q(e^{it})=-1.
		$
		
		Choose continuous arguments
		\[
		\mathcal B_p(e^{it})=e^{i\phi_p(t)},
		\qquad
		\mathcal B_q(e^{it})=e^{i\phi_q(t)},
		\]
		normalized by
		$
		\phi_p(0)=\phi_q(0)=0.
		$
		By the discussion following \eqref{eq:Bf}, $\mathcal B_p$ and
		$\mathcal B_q$ are finite Blaschke products of degrees $n$ and $n+1$,
		respectively. Hence, by Lemma~\ref{lem:blaschke}, $\phi_p$ and $\phi_q$
		are strictly increasing and
		\[
		\phi_p(2\pi)=\phi_p(2\pi)-\phi_p(0)=2\pi n,\qquad
		\phi_q(2\pi)=\phi_q(2\pi)-\phi_q(0)=2\pi(n+1).
		\]
		Proposition~\ref{prop:noncoincidence} gives
		\[
		\phi_q(t)-\phi_p(t)\notin2\pi\mathbb Z,
		\qquad 0<t<2\pi.
		\]
		Moreover,
		\[
		\phi_q(0)-\phi_p(0)=0,
		\qquad
		\phi_q(2\pi)-\phi_p(2\pi)=2\pi.
		\]
		Since $\phi_q-\phi_p$ is continuous, it follows that
		\begin{equation}\label{eq:phase-separation}
			0<\phi_q(t)-\phi_p(t)<2\pi,
			\qquad 0<t<2\pi.
		\end{equation}
		
		Let
		\[
		p^{\#}(e^{i\theta_k})=0, \qquad q^{\#}(e^{i\tau_k})=0,
		\]
		where the arguments are listed in increasing order. By \eqref{eq:root-level},
		\[
		\phi_p(\theta_k)=(2k-1)\pi,
		\qquad
		\phi_q(\tau_k)=(2k-1)\pi.
		\]
		Taking $t=\theta_k$ in \eqref{eq:phase-separation}, we obtain
		\[
		\phi_q(\tau_k)=(2k-1)\pi
		=\phi_p(\theta_k)
		<\phi_q(\theta_k)
		<\phi_p(\theta_k)+2\pi
		=(2k+1)\pi
		=\phi_q(\tau_{k+1}).
		\]
		Since $\phi_q$ is strictly increasing,
		\[
		\tau_k<\theta_k<\tau_{k+1},
		\qquad 1\le k\le n.
		\]
		Hence
		\[
		0<\tau_1<\theta_1<\tau_2<\cdots
		<\tau_n<\theta_n<\tau_{n+1}<2\pi.
		\]
		This completes the proof.
	\end{proof}

	\begin{remark}
		The assumptions that $p(z)$ and $q(z)$ have only negative real zeros in Theorem~\ref{thm:main} can be replaced
		by a weaker coefficient condition. More precisely, let
		\[
		p(z)=\sum_{k=0}^{n}p_kz^k\in\mathcal{P}_n,
		\qquad
		q(z)=\sum_{k=0}^{n+1}q_kz^k\in\mathcal{P}_{n+1},
		\]
		where $n\ge 1$, and suppose that the two normalized coefficient
		sequences
		\[
		\left(\frac{p_k}{\binom{n}{k}}\right)_{k=0}^{n}
		\qquad\text{and}\qquad
		\left(\frac{q_k}{\binom{n+1}{k}}\right)_{k=0}^{n+1}
		\]
		are log-concave. Then all zeros of $p^\#$ and $q^\#$ are
		simple and lie on the unit circle, and $p^\#$ strictly
		interlaces $q^\#$.
		
		Indeed, let
		$f(z)=\sum_{k=0}^{n}a_kz^k\in\mathcal{P}_n$, where $n\ge2$,
		and assume that its normalized coefficient sequence is log-concave.
		By scaling, we may suppose that $a_0=a_n=1$.
		The normalized coefficient sequence is symmetric and log-concave,
		and hence is nondecreasing up to the middle. Thus
		\[
		a_1\ge n,
		\qquad
		\frac{a_k}{a_{k-1}}
		\ge \frac{n-k+1}{k}>1,
		\qquad
		1\le k\le \left\lfloor\frac n2\right\rfloor.
		\]
		Consequently, the conclusions of Lemmas~2.3 and~3.1 remain valid.
		The subsequent sector and phase arguments apply unchanged.
		The case $n=1$ follows directly.
	\end{remark}
	
	\section{Applications}\label{sec:applications}
	In this section, we derive some applications of Theorem~\ref{thm:main}.
	
	\subsection{Classical families of polynomials with only real zeros}
	
	Many classical polynomial families arising in enumerative combinatorics
	have positive palindromic coefficients and only negative real zeros.
	We consider here three standard examples: the binomial polynomials
	$B_n(z)$, the Eulerian polynomials $E_n(z)$, and the Narayana polynomials
	$\mathcal N_n(z)$, namely,
	\[
	\begin{aligned}
		B_n(z)&=(1+z)^n,\\
		E_n(z)&=\sum_{k=0}^{n}
		\left\langle {n+1\atop k}\right\rangle z^k,\\
		\mathcal N_n(z)
		&=\sum_{k=0}^{n}
		\frac{1}{n+1}
		\binom{n+1}{k+1}
		\binom{n+1}{k}z^k,
	\end{aligned}
	\]
	where $n\geq0$ and $\left\langle {n+1\atop k}\right\rangle$ denotes the Eulerian number counting permutations of $[n+1]$ with $k$ descents.
	The reciprocal binomial coefficients have also been studied extensively,
	in particular from the viewpoint of identities, sums, and generating
	functions; see, for example, \cite{KruchininKruchinin2026} and the
	references therein.
	
	For $n\ge 1$, the polynomial $B_n$ has the single zero $-1$, with multiplicity $n$.
	It is well known~\cite{LiuWang2007} that the Eulerian and Narayana polynomials have only real zeros. Moreover, each of the three families
	has positive palindromic coefficients, with constant and leading
	coefficients equal to $1$. Hence
	\[
	B_n,\ E_n,\ \mathcal N_n\in\mathcal P_n \cap\RZ.
	\]
	
	Applying Theorem~\ref{thm:main}, we immediately obtain the following
	consequence.
	
	\begin{corollary}\label{cor:classical}
		For every $n\geq1$,
		\begin{enumerate}[(i)]
			\item $B_n^{\#}$ strictly interlaces $B_{n+1}^{\#}$;
			\item $E_n^{\#}$ strictly interlaces $E_{n+1}^{\#}$;
			\item $\mathcal N_n^{\#}$ strictly interlaces
			$\mathcal N_{n+1}^{\#}$.
		\end{enumerate}
	\end{corollary}

	We note that Theorem~\ref{thm:main} does not require the two polynomials
	to belong to the same family. Thus, a degree-$n$ polynomial from any one
	of the three families above may also be paired with a degree-$(n+1)$
	polynomial from either of the other two families.
	
	\subsection{A two-step construction from unit-circle polynomials}
	
	It is known~\cite{Pet15} that every real palindromic polynomial $f$ of degree $n$ has a unique $\gamma$-expansion
	\[
	f(z)=\sum_{j=0}^{\lfloor n/2\rfloor}
	\gamma_j z^j(1+z)^{n-2j}.
	\]
	Define the \emph{$\gamma$-sign transform} of $f$ by
	\begin{equation}\label{eq:gamma-sign-transform}
		\mathcal G[f](z)
		=
		\sum_{j=0}^{\lfloor n/2\rfloor}
		(-1)^j\gamma_j z^j(1+z)^{n-2j}.
	\end{equation}
	
	\begin{lemma}\label{lem:unitcircle-to-real}
		Let $f\in\UC$ be a real palindromic polynomial of degree $n$ with
		$f(0)>0$. Then
		\[
		\mathcal G[f]\in\mathcal P_n\cap\RZ.
		\]
	\end{lemma}
	
	\begin{proof}
		Write
		\[
		n=2m+\varepsilon,
		\qquad
		\varepsilon\in\{0,1\}.
		\]
		Since $f$ has real coefficients and all its zeros lie on the unit circle,
		its nonreal zeros occur in conjugate pairs. Moreover, palindromicity
		implies that the zero $1$, if present, has even multiplicity, while
		$-1$ occurs with odd multiplicity when $n$ is odd. Hence
		\[
		f(z)
		=
		f(0)(1+z)^{\varepsilon}
		\prod_{\ell=1}^{m}
		(1+b_{\ell}z+z^2),
		\qquad
		-2\le b_{\ell}\le2.
		\]
		Each quadratic factor can be written as
		$
		1+b_{\ell}z+z^2
		=
		(1+z)^2-(2-b_{\ell})z
		$.
		Therefore, expanding the product in the $\gamma$-basis and replacing
		each $\gamma_j$ by $(-1)^j\gamma_j$ changes the minus sign in each
		quadratic factor to a plus sign. Thus
		\begin{align*}
			\mathcal G[f](z)
			&=
			f(0)(1+z)^{\varepsilon}
			\prod_{\ell=1}^{m}
			\bigl((1+z)^2+(2-b_{\ell})z\bigr)\\
			&=
			f(0)(1+z)^{\varepsilon}
			\prod_{\ell=1}^{m}
			\bigl(1+(4-b_{\ell})z+z^2\bigr).
		\end{align*}
		Since
		$
		2\le4-b_{\ell}\le6
		$,
		each quadratic factor has two negative real zeros, possibly coinciding
		at $-1$.
		Since $f(0)>0$, all factors have positive coefficients, and their product is palindromic with only real zeros.
	\end{proof}
	
	Combining Lemma~\ref{lem:unitcircle-to-real} with
	Theorem~\ref{thm:main} gives the following two-step construction.
	
	\begin{corollary}\label{cor:two-step}
		Let $f,g\in\UC$ be real palindromic polynomials of degrees $n$ and
		$n+1$, respectively, where $n\ge1$, and suppose that
		$f(0)>0$, $g(0)>0$.
		Then
		\[
		\bigl(\mathcal G[f]\bigr)^{\#},
		\qquad
		\bigl(\mathcal G[g]\bigr)^{\#}
		\in\UC,
		\]
		and
		\[
		\bigl(\mathcal G[f]\bigr)^{\#}
		\quad\text{strictly interlaces}\quad
		\bigl(\mathcal G[g]\bigr)^{\#}.
		\]
	\end{corollary}
	
	Notice that no interlacing relation between the zeros of $f$ and $g$
	is required in Corollary~\ref{cor:two-step}. Thus the two-step map
	\[
	f\longmapsto \mathcal G[f]
	\longmapsto \bigl(\mathcal G[f]\bigr)^{\#}
	\]
	produces strictly interlacing unit-circle polynomials from independently
	chosen unit-circle polynomials of consecutive degrees.
	
	\begin{example}
		\label{ex:rogers-szego}
		For $0<t<1$, let
		\[
		R_n(z;t)
		=
		\sum_{k=0}^{n}\binom{n}{k}_{t}z^k,
		\qquad
		\binom{n}{k}_{t}
		=
		\prod_{j=1}^{k}
		\frac{1-t^{\,n-k+j}}{1-t^j},
		\]
		with the endpoint cases defined by continuity. Thus
		\[
		R_n(z;0)=1+z+\cdots+z^n,
		\qquad
		R_n(z;1)=(1+z)^n.
		\]
		The Rogers--Szeg\H{o} polynomials are real and palindromic, and all
		their zeros lie on the unit circle. For $0<t<1$, this follows from
		their relation with the continuous $q$-Hermite polynomials
		\cite{Ismail2009}.
		
		Let
		\[
		\widehat R_n(z;t)=\mathcal G[R_n(z;t)].
		\]
		The Rogers--Szeg\H{o} recurrence
		\[
		R_{n+1}(z;t)
		=
		(1+z)R_n(z;t)
		-
		(1-t^n)zR_{n-1}(z;t)
		\]
		is transformed into
		\begin{equation}\label{eq:RS-transformed-recurrence}
			\widehat R_{n+1}(z;t)
			=
			(1+z)\widehat R_n(z;t)
			+
			(1-t^n)z\widehat R_{n-1}(z;t),
		\end{equation}
		with
		\[
		\widehat R_0(z;t)=1,
		\qquad
		\widehat R_1(z;t)=1+z.
		\]
		
		Define
		\[
		T_n(z;t)
		=
		\bigl(\widehat R_n(z;t)\bigr)^{\#}.
		\]
		Since $R_n(z;t)\in\UC$ for every $t\in[0,1]$,
		Corollary~\ref{cor:two-step} gives, for arbitrary
		$s,t\in[0,1]$,
		\[
		T_n(z;t)
		\quad\text{strictly interlaces}\quad
		T_{n+1}(z;s),
		\qquad n\ge1.
		\]
		In particular, the parameters in two consecutive degrees may be
		chosen independently.
	\end{example}
	
	\begin{remark}
		For a fixed $0<t<1$, the original Rogers--Szeg\H{o} polynomials
		$R_n(z;t)$ and $R_{n+1}(z;t)$ themselves strictly interlace on the
		unit circle, as follows from the interlacing of consecutive continuous
		$q$-Hermite polynomials. The conclusion of Example~\ref{ex:rogers-szego}
		is different: it concerns the transformed family $T_n$ and allows the
		parameters in two consecutive degrees to be chosen independently.
		
		The two endpoints also recover familiar families. When $t=1$,
		\[
		\widehat R_n(z;1)=(1+z)^n,
		\]
		so $T_n(z;1)$ is the reciprocal-binomial polynomial. When $t=0$,
		\[
		\widehat R_{n+1}(z;0)
		=
		(1+z)\widehat R_n(z;0)
		+
		z\widehat R_{n-1}(z;0),
		\]
		which is the recurrence for the symmetric Delannoy row polynomials \cite{WangZhengChen2019}.
		Thus the family $T_n(z;t)$ connects the reciprocal Delannoy and
		reciprocal-binomial families.
	\end{remark}
	
	\begin{example}
		\label{ex:poupard-kreweras}
		
		Let $F_m(z)$ and $K_m(z)$ be defined by
		\[
		F_1(z)=1,
		\qquad
		(z-1)^2F_{m+1}(z)
		=
		(z^{2m+2}+1)F_m(1)-2z^2F_m(z),
		\]
		and
		\[
		K_1(z)=1+z,
		\qquad
		(z-1)^2K_{m+1}(z)
		=
		(z^{2m+3}+1)K_m(1)-2z^2K_m(z),
		\]
		for $m\ge1$.
		These are the Poupard and Kreweras-related polynomials studied by
		Chapoton and Han~\cite{ChapotonHan2020}. Their zeros lie on the unit
		circle, and
		\[
		\deg F_m=2m-2,
		\qquad
		\deg K_m=2m-1.
		\]
		Hence, these two families provide one polynomial of each nonnegative degree.
		
		Set
		\[
		p_{2m-2}(z)
		=
		\mathcal G\!\left[\frac{F_m(z)}{F_m(0)}\right],
		\qquad
		p_{2m-1}(z)
		=
		\mathcal G\!\left[\frac{K_m(z)}{K_m(0)}\right],
		\]
		and define
		\[
		T_n(z)=p_n^{\#}(z).
		\]
		By Lemma~\ref{lem:unitcircle-to-real},
		\[
		p_n\in\mathcal P_n\cap\RZ.
		\]
		Therefore Theorem~\ref{thm:main} gives
		\[
		T_n
		\quad\text{strictly interlaces}\quad
		T_{n+1},
		\qquad n\ge1.
		\]
		Thus the two unit-circle families are combined, through the two-step
		construction, into a single strictly interlacing sequence indexed by
		consecutive degrees.
	\end{example}

	\section{Remarks}\label{Remark}
	In this section, we show that our main theorem also yields a strict
	interlacing result on the critical line through the
	Rodr\'iguez--Villegas transformation.
	
	Polynomials whose zeros lie on the critical line
	\[
	\operatorname{Re} z=-\frac12
	\]
	arise naturally in the study of Hilbert and Ehrhart polynomials; see,
	for example, \cite{HigashitaniKummerMichaelek2017,Kolbl2025}.
	Rodr\'iguez--Villegas \cite{RodriguezVillegas2002} showed that a
	unit-circle zero condition for the numerator of a rational generating
	function forces the zeros of the associated polynomial to lie on a
	critical line.
	
	More precisely, let
	\[
	h(x)=\sum_{k=0}^{n}h_kx^k,
	\qquad h_n\neq0.
	\]
	Define its Rodr\'iguez--Villegas transform by
	\[
	\mathcal R_n[h](z)
	=
	\sum_{k=0}^{n}
	h_k\binom{z+n-k}{n}.
	\]
	Equivalently,
	\[
	\sum_{m\geq0}\mathcal R_n[h](m)x^m
	=
	\frac{h(x)}{(1-x)^{n+1}}.
	\]
	Rodr\'iguez--Villegas proved that if all zeros of $h$ lie on the unit
	circle and $h(1)\neq0$, then all zeros of $\mathcal R_n[h]$ lie on the line $\operatorname{Re}z=-\frac12$.
	Rodr\'iguez further proved that the Rodr\'iguez--Villegas transformation
	preserves strict interlacing from the unit circle to the corresponding critical
	line; see \cite[Theorem~2.1.10]{Rodriguez2010}.
	Interlacing on the line \(\operatorname{Re}z=-1/2\) is defined by ordering the zeros according to their imaginary parts.
	
	Theorem~\ref{thm:main} supplies the required unit-circle interlacing for any independently chosen pair of consecutive degrees in $\mathcal P_n\cap\RZ$ and $\mathcal P_{n+1}\cap\RZ$.
	Since $p^\#(1)>0$ and $q^\#(1)>0$, the nonvanishing
	condition at $1$ is satisfied.
	Thus Rodr\'iguez's result yields the following critical-line counterpart
	of our main theorem.
	
	\begin{corollary}\label{cor:critical-line}
		Let
		\[
		p\in\mathcal P_n\cap\RZ,
		\qquad
		q\in\mathcal P_{n+1}\cap\RZ,
		\]
		where $n\geq1$. Then
		$\mathcal R_n[p^{\#}]$
		strictly interlaces
		$\mathcal R_{n+1}[q^{\#}]$
		on the line $\operatorname{Re} z=-\frac12$.
	\end{corollary}

	\section*{Declaration of generative AI and AI-assisted technologies in the manuscript preparation process}
	During the preparation of this work the authors used ChatGPT (OpenAI) to assist with language editing and LaTeX formatting. After using this tool/service, the authors reviewed and edited the content as needed and take full responsibility for the content of the published article.


\begin{thebibliography}{99}
		
		
		
		\bibitem{BorceaBranden2009Circular}
		J.~Borcea and P.~Br\"and\'en,
		P\'olya--Schur master theorems for circular domains and their boundaries,
		\emph{Ann. of Math. (2)} \textbf{170} (2009), 465--492.
		
		\bibitem{Brenti1989}
		F.~Brenti,
		Unimodal, log-concave and P\'olya frequency sequences in combinatorics,
		\emph{Mem. Amer. Math. Soc.} \textbf{81} (1989), no.~413.
		
		
		\bibitem{BY01}
		G.~Brown and Q.~Yin,
		Positivity of a class of cosine sums,
		\emph{Acta Sci. Math. (Szeged)} \textbf{67} (2001), 221--247.
		
		\bibitem{ChapotonHan2020}
		F.~Chapoton and G.-N.~Han,
		On the roots of the Poupard and Kreweras polynomials,
		\emph{Moscow J. Combin. Number Theory} \textbf{9} (2020), 163--172.
		
		\bibitem{Chen1995}
		W.~Chen,
		On the polynomials with all their zeros on the unit circle,
		\emph{J. Math. Anal. Appl.} \textbf{190} (1995), 714--724.
		
		\bibitem{ChenMaoWangXiao2026}
		X.~Chen, J.~Mao, Y.~Wang, and Q.~Xiao,
		Palindromic polynomials with alternating gamma coefficients,
		preprint, 2026.
		
		\bibitem{GarciaMashreghiRoss2017}
		S.~R.~Garcia, J.~Mashreghi, and W.~T.~Ross,
		Finite Blaschke products: a survey,
		in: \emph{Harmonic Analysis, Function Theory, Operator Theory, and Their Applications},
		Theta Ser. Adv. Math., vol.~19, 2017, pp.~133--158.
		
		\bibitem{GardnerGovil2014}
		R.~B.~Gardner and N.~K.~Govil,
		Enestr\"om--Kakeya theorem and some of its generalizations,
		in: S.~Joshi, M.~Dorff, and I.~Lahiri (eds.),
		\emph{Current Topics in Pure and Computational Complex Analysis},
		Trends in Mathematics,
		Birkh\"auser/Springer, 2014, pp.~171--199.
		
		\bibitem{HigashitaniKummerMichaelek2017}
		A.~Higashitani, M.~Kummer, and M.~Micha{\l}ek,
		Interlacing Ehrhart polynomials of reflexive polytopes,
		\emph{Selecta Math. (N.S.)} \textbf{23} (2017), 2977--2998.
		
		\bibitem{Ismail2009}
		M.~E.~H.~Ismail,
		\emph{Classical and Quantum Orthogonal Polynomials in One Variable},
		Encyclopedia of Mathematics and its Applications, vol.~98,
		Cambridge University Press, Cambridge, 2009.
		
		\bibitem{Kolbl2025}
		M.~K\"olbl,
		Properties of Ehrhart polynomials whose roots lie on the canonical line,
		\emph{Integers} \textbf{25} (2025), Paper No.~A68, 19~pp.
		
		\bibitem{KruchininKruchinin2026}
		D.~Kruchinin and V.~Kruchinin,
		A family of generating functions for reciprocal binomial coefficients and its applications,
		preprint, arXiv:2602.07822, 2026.
		
		\bibitem{LakatosLosonczi2007}
		P.~Lakatos and L.~Losonczi,
		Circular interlacing with reciprocal polynomials,
		\emph{Math. Inequal. Appl.} \textbf{10} (2007), no.~4, 761--769.
		
		\bibitem{LiuWang2007}
		L.~L.~Liu and Y.~Wang,
		A unified approach to polynomial sequences with only real zeros,
		\emph{Adv. in Appl. Math.} \textbf{38} (2007), 542--560.
		
		\bibitem{MandelRobins2019}
		A.~Mandel and S.~Robins,
		Dragging the roots of a polynomial to the unit circle,
		preprint, arXiv:1908.03208, 2019.
		
		\bibitem{Marden1966}
		M.~Marden,
		\emph{Geometry of Polynomials}, 2nd ed.,
		Mathematical Surveys, vol.~3,
		American Mathematical Society, Providence, RI, 1966.
		
		
		\bibitem{Obreschkoff1963}
		N.~Obreschkoff,
		\emph{Verteilung und Berechnung der Nullstellen reeller Polynome},
		Hochschulb\"ucher f\"ur Mathematik, vol.~55,
		Deutscher Verlag der Wissenschaften, Berlin, 1963.
		
		\bibitem{Pet15}
		T.~K.~Petersen,
		\emph{Eulerian Numbers},
		Birkh\"auser Advanced Texts: Basler Lehrb\"ucher,
		Birkh\"auser/Springer, New York, 2015.
		
		\bibitem{RahmanSchmeisser2002}
		Q.~I.~Rahman and G.~Schmeisser,
		\emph{Analytic Theory of Polynomials},
		London Mathematical Society Monographs, New Series, vol.~26,
		Oxford University Press, Oxford, 2002.
		
		\bibitem{Rodriguez2010}
		M.~A.~Rodr\'iguez,
		\emph{The Distribution of Roots of Certain Polynomials},
		Ph.D. dissertation, The University of Texas at Austin, 2010.
		
		\bibitem{RodriguezVillegas2002}
		F.~Rodr\'iguez-Villegas,
		On the zeros of certain polynomials,
		\emph{Proc. Amer. Math. Soc.} \textbf{130} (2002), no.~8, 2251--2254.
		
		\bibitem{Simanek2020}
		B.~Simanek,
		Zero spacings of paraorthogonal polynomials on the unit circle,
		\emph{J. Approx. Theory} \textbf{256} (2020), Paper No.~105437.
		
		\bibitem{Simon2007}
		B.~Simon,
		Rank one perturbations and the zeros of paraorthogonal polynomials on the unit circle,
		\emph{J. Math. Anal. Appl.} \textbf{329} (2007), 376--382.
		
		\bibitem{WangZhengChen2019}
		Y.~Wang, S.-N.~Zheng, and X.~Chen,
		Analytic aspects of Delannoy numbers,
		\emph{Discrete Math.} \textbf{342} (2019), no.~8, 2270--2277.
		
	\end{thebibliography}
\end{document}